\documentclass[a4paper,english]{lipics-v2018}
\usepackage{microtype}
\usepackage{tikz}
\nolinenumbers
\hideLIPIcs

\theoremstyle{plain}
\newtheorem{conjecture}[theorem]{Conjecture}

\newcommand{\R}{\mathbb R}
\newcommand{\eps}{\varepsilon}
\newcommand{\prefrel}{\sqsubseteq}
\newcommand{\sol}{\mathrm{Sol}}
\newcommand{\wordsum}{\Sigma}
\newcommand{\psw}{\mathrm{psw}}
\newcommand{\area}[1]{A(#1)}
\newcommand{\height}[1]{H(#1)}

\title{An Optimal Bound on the Solution Sets of~One-Variable Word Equations and its Consequences}
\titlerunning{Solutions of One-Variable Word Equations}

\author{Dirk Nowotka}{Department of Computer Science, Kiel University, 24098 Kiel, Germany}{dn@informatik.uni-kiel.de}{}{This work was partially supported by the DFG research project 181615770}
\author{Aleksi Saarela}{Department of Mathematics and Statistics, University of Turku, 20014 Turku, Finland}{amsaar@utu.fi}{ https://orcid.org/0000-0002-6636-2317}{}
\authorrunning{D. Nowotka and A. Saarela}
\Copyright{Dirk Nowotka and Aleksi Saarela}

\subjclass{}

\EventEditors{Ioannis Chatzigiannakis, Christos Kaklamanis, D\'{a}niel Marx, and Don Sannella}
\EventNoEds{4}
\EventLongTitle{45th International Colloquium on Automata, Languages, and Programming (ICALP 2018)}
\EventShortTitle{ICALP 2018}
\EventAcronym{ICALP}
\EventYear{2018}
\EventDate{July 9--13, 2018}
\EventLocation{Prague, Czech Republic}
\EventLogo{eatcs}
\SeriesVolume{107}
\ArticleNo{225}

\begin{document}

\maketitle

\begin{abstract}
We solve two long-standing open problems on word equations.
Firstly, we prove that a one-variable word equation with constants
has either at most three or an infinite number of solutions.
The existence of such a bound had been conjectured,
and the bound three is optimal.
Secondly, we consider independent systems of three-variable word equations
without constants.
If such a system has a nonperiodic solution,
then this system of equations is at most of size 17.
Although probably not optimal, this is the first finite bound found.
However, the conjecture of that bound being actually two still remains open.
\end{abstract}

\section{Introduction}

If $n$ words satisfy a nontrivial relation,
they can be written as products of $n - 1$ words.
This folklore result is known as the defect theorem,
and it can be seen as analogous to the simple fact of linear algebra
that the dimension of the solution space
of a homogeneous $n$-variable linear equation is $n - 1$.
If an independent equation is added to a system of linear equations,
the dimension of the solution space decreases,
which gives an upper bound $n$
for the size of~independent systems of linear equations,
but no such results are known for word equations.
In fact, the maximal size of independent systems of constant-free word equations
has been one of the biggest open questions in combinatorics on words
for many decades.
In 1983, Culik and Karhum\"aki~\cite{CuKa83} pointed out
that a conjecture of Ehrenfeucht about test sets of formal languages
can be equivalently formulated as claiming that
every infinite system of word equations is equivalent to a finite subsystem.
Ehrenfeucht's conjecture was proved
by Albert and Lawrence~\cite{AlLa85ehrenfeucht}
and independently by Guba~\cite{Gu86},
and it follows that independent systems cannot be infinite,
but no finite upper bounds depending only on the number of variables
have been found.
Independent systems of size $\Theta(n^4)$ on $n$ variables
were constructed by Karhum\"aki and Plandowski~\cite{KaPl94},
and the hidden constant in $\Theta(n^4)$ was improved in \cite{KaSa11adian}.
This is the best known lower bound.

The case of three variables is particularly interesting.
In this case,
it is easy to find systems of size two
that are independent and have a nonperiodic solution,
or systems of size three
that are independent but have no nonperiodic solution,
and Culik and Karhum\"aki conjectured that there are no larger such systems,
but no finite upper bounds have been found even in this case.
In fact, despite Ehrenfeucht's conjecture,
even the existence of a~bound is not guaranteed,
because in principle it might be possible
that there are unboundedly large finite independent systems.
This case of three variables is very striking
because it is the simplest nontrivial case,
but the gap between the almost trivial lower bound and the infinite upper bound
has remained huge despite the considerable attention the problem has received.
Some results about systems of specific forms
are known~\cite{HaNo03,CzKa07,CzPl09},
and some upper bounds that depend on the sizes of the equations
have been proved~\cite{Sa15ejc,HoZe15,NoSa18ijfcs}.
The best current bound is logarithmic with respect to
the size of the smallest equation in the system~\cite{NoSa18ijfcs}.

In the above, we have considered constant-free word equations.
If we add constants, the equations become more complicated.
For constant-free equations,
the three-variable case is the first nontrivial one,
but for equations with constants,
already the one-variable case is interesting.
One-variable equations
have been studied in many articles~\cite{EyGoMa94,DaPl11,LaPl11},
and the main open question about them
is the maximal number of solutions such an equation can have
if we exclude equations with infinitely many solutions
(if the solution set is infinite, it is known to be of a very specific form).
Even finding an example with exactly two solutions is not entirely trivial,
but a simple example was given by Laine and Plandowski~\cite{LaPl11}.
An~example with exactly three solutions was recently found~\cite{NoSa18ijfcs}.
No fixed upper bound, or even the existence of an upper bound, has been proved.
The best known result is a bound that depends logarithmically
on the number of occurrences of the variable in the equation~\cite{LaPl11}.
It can be noted that the solutions of a one-variable equation
can be found in linear time in the RAM model, as proved by Je\.z~\cite{Je16one}.

In this article,
we solve the open problem about sizes of solution sets of one-variable equations
by proving that a one-variable equation
has either infinitely many solutions or at most three,
which is an optimal result.
As a consequence,
we prove the first upper bound for the sizes of independent systems
of constant-free three-variable equations,
thus settling the old open question about the existence of such a bound.
More specifically,
we prove that if an independent system
of constant-free three-variable equations
is independent and has a~nonperiodic solution,
then the system is of size at most 17
(if the system is not required to have a~nonperiodic solution,
then the size can be at most one larger).
This bound is probably not optimal
and the conjecture of Culik and Karhum\"aki remains open,
as does the more general question about $n$-variable equations.

Two previous articles provide crucial tools for our proofs.
The first article is~\cite{Sa17stacs},
where new methods were introduced
to solve a certain open problem on word equations.
We use and further develop these methods to analyze one-variable equations.
The second article is~\cite{NoSa18ijfcs},
where a surprising connection
between the two topics we have discussed above was found:
It was proved that a bound for the maximal size
of a finite solution set of a one-variable equation
implies a (larger) bound for the maximal size of independent systems
of constant-free three-variable equations.

\section{Preliminaries}

We begin this section by considering constant-free word equations.
Let $\Xi$ be an alphabet of variables and $\Gamma$ an alphabet of constants.
A \emph{constant-free word equation}
is a pair $(U, V) \in \Xi^* \times \Xi^*$,
and the solutions of this equation are the morphisms
\begin{math}
    h: \Xi^* \to \Gamma^*
\end{math}
such that $h(U) = h(V)$.
A solution $h$ is \emph{periodic}
if there exists $p \in \Gamma^*$ such that $h(X) \in p^*$ for all $X \in \Xi$.
Otherwise, $h$ is \emph{nonperiodic}.
It is well-known that $h$ is periodic if and only if
$h(PQ) = h(QP)$ for all words $P, Q \in \Xi^*$.

\begin{example} \label{exa:equations1}
    Let $\Xi = \{X, Y, Z\}$ and consider the equation $(XYZ, ZYX)$.
    For all  $p, q \in \Gamma^*$ and $i, j, k \geq 0$,
    the morphism $h$ defined by
    $h(X) = (pq)^i p$, $h(Y) = (qp)^j q$, $h(Z) = (pq)^k p$
    is a~solution of this equation because
    \begin{equation*}
        h(XYZ) = (pq)^i p \cdot (qp)^j q \cdot (pq)^k p = (pq)^{i + j + k + 1} p
        = (pq)^k p \cdot (qp)^j q \cdot (pq)^i p = h(ZYX).
    \end{equation*}
    Every nonperiodic solution of the equation is of this form.
\end{example}

A set of equations is a \emph{system of equations}.
A morphism is a solution of a system
if it is a solution of every equation in the system.
Two equations or systems are \emph{equivalent}
if they have exactly the same solutions.
A system of equations is \emph{independent}
if it is not equivalent to any of its proper subsets.

\begin{example} \label{exa:cert}
    Let $\Xi = \{X, Y, Z\}$ and $\Gamma = \{a, b\}$.
    The system of equations
    $S = \{(\mathit{XYZ}, \mathit{ZYX})$, $(XYYZ, ZYYX)\}$
    is independent and has a nonperiodic solution $h$ defined by
    $h(X) = a$, $h(Y) = b$, $h(Z) = a$.
    To see independence, note that $S$ is not equivalent to $(XYZ, ZYX)$,
    because the morphism $h$ defined by
    $h(X) = a$, $h(Y) = b$, $h(Z) = aba$
    is a solution of $(XYZ, ZYX)$ but not of $S$,
    and $S$ is not equivalent to $(XYYZ, ZYYX)$,
    because the morphism $h$ defined by
    $h(X) = a$, $h(Y) = b$, $h(Z) = abba$
    is a solution of $(XYYZ, ZYYX)$ but not of $S$.
\end{example}

The following question is a big open problem on word equations:
If a system of constant-free three-variable equations
is independent and has a nonperiodic solution,
then how large can the system be?
The largest known examples are of size two, see Example~\ref{exa:cert},
and it has been conjectured that these examples are optimal.
Even the following weaker conjecture is open.

\begin{conjecture} \label{con:threevar}
    There exists a number $c$ such that
    every independent system of constant-free three-variable equations
    with a nonperiodic solution
    is of size $c$ or less.
\end{conjecture}

Currently, the best known result is the following.

\begin{theorem}[\cite{NoSa18ijfcs}] \label{thm:three_ind_log}
    Every independent system of constant-free three-variable equations
    is of size $O(\log n)$,
    where $n$ is the length of the shortest equation.
\end{theorem}

Next, we will consider word equations with constants.
As before,
let $\Xi$ be an alphabet of variables and $\Gamma$ an alphabet of constants.
A \emph{word equation with constants}
is a pair $(U, V) \in (\Xi \cup \Gamma)^* \times (\Xi \cup \Gamma)^*$,
and the solutions of this equation are the constant-preserving morphisms
\begin{math}
    h: (\Xi \cup \Gamma)^* \to \Gamma^*
\end{math}
such that $h(U) = h(V)$.
If $U = V$, then the equation is \emph{trivial}.

In this article, we are interested in the one-variable case $\Xi = \{X\}$.
We use the notation $[u]$
for the constant-preserving morphism $h: (\{X\} \cup \Gamma)^* \to \Gamma^*$
defined by $h(X) = u$.
If $S$ is a set of words, we use the notation $[S] = \{[u] \mid u \in S\}$.
If $[u]$ is a solution of a one-variable equation $E$,
then $u$ is called a \emph{solution word} of $E$.
The set of all solutions of $E$ is denoted by $\sol(E)$.

\begin{example} \label{exa:equations2}
    Let $\Gamma = \{a, b\}$.
    The equation $(Xab, abX)$ has infinitely many solutions
    $[(ab)^i]$, where $i \geq 0$.
    The equation $(XaXbab, abaXbX)$ has exactly two solutions
    $[\eps]$ and $[ab]$.
    The equation $(XXbaaba, aabaXbX)$ has exactly two solutions
    $[a]$ and $[aaba]$.
    The equation
    \begin{equation*}
        (X aXb X aabb ab aXb abaabbab, abaabbab aXb ab aabb X aXb X)
    \end{equation*}
    has exactly three solutions $[\eps], [ab], [abaabbab]$.
\end{example}

The following is a well-known open problem:
If a one-variable equation has only finitely many solutions,
then what is the maximal number of solutions it can have?
Example~\ref{exa:equations2} shows that the answer is at least three,
but no upper bound is known.
Currently, the best known result is the following.

\begin{theorem}[{\cite[Theorems 23, 26, 29]{LaPl11}}] \label{thm:one_sol_log}
    If the solution set of a one-variable equation is finite,
    then it has size at most $8 \log n + O(1)$,
    where $n$ is the number of occurrences of the variable.

    If the solution set is infinite and the equation is not trivial,
    then there are words $p, q$ such that $pq$ is primitive
    and the solution set is $[(pq)^* p]$.
\end{theorem}

We will need the following lemma.

\begin{lemma}[{\cite[Lemma 1]{DaPl11}}] \label{lem:one_pq}
    Let $E$ be a one-variable equation
    and let $pq$ be primitive.
    The set
    \begin{equation*}
        \sol(E) \cap [(pq)^+ p]
    \end{equation*}
    is either $[(pq)^+ p]$ or has at most one element.
\end{lemma}

A connection between constant-free three-variable equations
and one-variable equations with constants
was recently found~\cite{NoSa18ijfcs}.
Here we give the relevant special case of one of the results.

\begin{theorem}[\cite{NoSa18ijfcs}] \label{thm:one_three}
    If every one-variable word equation
    has either infinitely many solutions or at most three,
    then Conjecture~\ref{con:threevar} is true for $c = 17$.
\end{theorem}

In this article, we will prove that every one-variable word equation
has either infinitely many solutions or at most three,
and thus Conjecture~\ref{con:threevar} is true for $c = 17$.

\section{Sums of words}

In this section,
we will give some definitions and ideas
that will be used in our proofs.
Most of these were introduced in \cite{Sa17stacs}.

We can assume that the alphabet $\Gamma$ is a subset of $\R$.
Then we can define $\wordsum(w)$
to be the sum of the letters of a word $w \in \Gamma^*$,
that is, if $w = a_1 \dotsm a_n$ and $a_1, \dots, a_n \in \Gamma$,
then $\wordsum(w) = a_1 + \dots + a_n$.
Words $w$ such that $\wordsum(w) = 0$ are called \emph{zero-sum words}.
If $w$ is zero-sum, then the morphism $[w]$ can also be called zero-sum.
The largest and smallest letters in a word $w$
can be denoted by $\max(w)$ and $\min(w)$, respectively.

The \emph{prefix sum word} of $w = a_1 \dotsm a_n$ is the word
\begin{math}
    \psw(w) = b_1 \dotsm b_n,
\end{math}
where
\begin{math}
    b_i = \wordsum(a_1 \dotsm a_i)
\end{math}
for all $i$.
Of course, $\psw(w)$ is usually not a word over $\Gamma$,
but over some other alphabet.
The mapping $\psw$ is injective and length-preserving.
We also use the notation
\begin{math}
    \psw_r(w) = c_1 \dotsm c_n,
\end{math}
where $r \in \R$ and
\begin{math}
    c_i = b_i + r
\end{math}
for all $i$.

\begin{example}
    Let $w = bbcaac$, where $a = 1$, $b = 2$, and $c = -3$.
    We have $|w| = 6$, $\max(w) = 2$, and $\min(w) = -3$.
    Because $\wordsum(w) = 2 + 2 - 3 + 1 + 1 - 3 = 0$, $w$ is a~zero-sum word.
    The prefix sum word of $w$ is $\psw(w) = 2 4 1 2 3 0$,
    and $\max(\psw(w)) = 4$ and $\min(\psw(w)) = 0$.
\end{example}

For a word $w$, we define its
\emph{height} $\height{w}$
and \emph{area} $\area{w}$:
\begin{align*}
    \height{w} &= \max(\psw(w))
        = \max\{\wordsum(u) \mid \eps \ne u \prefrel w\}, \\
    \area{w} &= \wordsum(\psw(w)) = \sum_{u \prefrel w} \wordsum(u),
\end{align*}
where $u \prefrel w$ means that $u$ is a prefix of $w$.
For the empty word, $\height{\eps} = -\infty$ and $\area{\eps} = 0$.

These definitions have the following graphical interpretation:
A word $w = a_1 \dotsm a_n$ can be represented by a polygonal chain by
starting at the origin,
moving $a_1$ steps up, one step to the right,
$a_2$ steps up, one step to the right, and so on.
The end point of this curve is then $(|w|, \wordsum(w))$.
The biggest $y$-coordinate
(after the initial line segment starting at the origin)
is $\height{w}$.
The number $\area{w}$ is the area under the curve,
defined in the same way as a definite integral,
that is, parts below the $x$-axis count as negative areas.
See Figure \ref{fig:curve} for an example.

\begin{figure}[htb]
    \centering
    \begin{tikzpicture}[scale=1]
        \fill[fill=white!90!black]
            (0, 0)
            |- ++(1, 1) |- ++(1, 1) |- ++(1, 1)
            |- ++(1, -2) |- ++(1, -2) |- ++(1, 1)
            |- ++(1, 1) |- ++(0, -1);
        \draw[ultra thick]
            (0, 0) node{$\bullet$}
            |- ++(1, 1) |- ++(1, 1) |- ++(1, 1)
            |- ++(1, -2) |- ++(1, -2) |- ++(1, 1)
            |- ++(1, 1) node{$\bullet$}
                node[above right]{$(|w|, \wordsum(w))$};
        \draw[->] (-1, 0) -- (12, 0);
        \draw[->] (0, -1) -- (0, 4);
        \draw[dashed] (7, 0) -- (7, 1);
        \draw[dashed] (3, 3) -- (10, 3);
        \draw[<->] (10, 0) -- (10, 3);
        \draw (10, 1.5) node[right]{$\height{w}$};
        \draw (2, 1) node{$A_1$};
        \draw (4.5, -0.5) node{$A_2$};
        \draw (6.5, 0.5) node{$A_3$};
        \draw (6, 3.5) node{$\area{w} = A_1 - A_2 + A_3$};
    \end{tikzpicture}
    \caption{
        Representation of the word $w = aaabbaa$,
        where $a = 1$ and $b = -2$.
        We have $|w| = 7$, $\wordsum(w) = 1$,
        $\height{w} = 3$, and $\area{w} = 7$.
    } \label{fig:curve}
\end{figure}

\begin{lemma}
    For words $w_1, \dots, w_n$, we have
   \begin{align*}
        \wordsum(w_1 \dotsm w_n) &= \wordsum(w_1) + \dots + \wordsum(w_n), \\
        \psw(w_1 \dotsm w_n)
            &= \prod_{i = 1}^n \psw_{\wordsum(w_1 \dotsm w_{i - 1})}(w_i), \\
        \height{w_1 \dotsm w_n}
            &= \max\{\wordsum(w_1 \dotsm w_{i - 1}) + \height{w_i}
                \mid 1 \leq i \leq n\}, \\
        \area{w_1 \dotsm w_n}
            &= \sum_{i = 1}^n
                (\area{w_i} + \wordsum(w_1 \dotsm w_{i - 1}) |w_i|).
    \end{align*}
\end{lemma}

\begin{proof}
    Follows easily from the definitions.
\end{proof}

When studying words from a combinatorial point of view,
the choice of the alphabet is arbitrary
(except for the size of the alphabet),
so we can assign numerical values to the letters in any way we like,
as long as no two letters get the same value.
The next two lemmas show that, given any word $w$,
the alphabet can be normalized so that $w$ becomes a zero-sum word,
and every zero-sum word can be written as a product of minimal zero-sum words
in a unique way.

\begin{lemma}[{\cite[Lemma 3]{Sa17stacs}}] \label{lem:alphabetnormalization}
    Let $w \in \Gamma^*$.
    There exists an alphabet $\Delta$
    and an isomorphism $h: \Gamma^* \to \Delta^*$
    such that $h(w)$ is zero-sum.
\end{lemma}

\begin{lemma}[{\cite[Lemma 4]{Sa17stacs}}]
    The set of zero-sum words over $\Gamma$ is a free monoid.
\end{lemma}

\section{Equations in normal form}

If a one-variable equation has more occurrences of the variable
on the left-hand side than on the right-hand side, or vice versa,
then it is easy to see by a length argument
that it can have at most one solution.
Therefore every one-variable equation with more than one solution
can be written in the form
\begin{equation} \label{eq:normalform}
    (u_0 X u_1 \dotsm X u_n, v_0 X v_1 \dotsm X v_n),
\end{equation}
where $X$ is the variable, $n \geq 1$,
and $u_0, \dots, u_n, v_0, \dots, v_n$ are constant words.
Clearly, it must be $|u_0 \dotsm u_n| = |v_0 \dotsm v_n|$.
If the equation is nontrivial, $x_1, x_2$ are solution words,
and $|x_1| \leq |x_2|$,
then it is quite easy to see that $x_1$ is a prefix and a suffix of $x_2$.

We say that the equation \eqref{eq:normalform} is in \emph{normal form}
if the following conditions are satisfied:
\begin{enumerate}[(N1)]
    \item \label{item:emptyzero}
    It has the empty solution
    and at least one other zero-sum solution,
    \item \label{item:less}
    $|u_0 \dotsm u_i| < |v_0 \dotsm v_i|$
    for all $i \in \{0, \dots, n - 1\}$,
    \item \label{item:leq}
    $|u_0 \dotsm u_i| \leq |v_0 \dotsm v_{i - 1}|$
    for all $i \in \{0, \dots, n\}$.
\end{enumerate}
It follows from these conditions that $u_0 = v_n = \eps$.
By the next two lemmas,
it is usually sufficient to consider equations in normal form.

\begin{lemma} \label{lem:assumeempty}
    Let $E$ be a one-variable equation,
    \begin{math}
        \sol(E) = \{[x_0], \dots, [x_m]\},
    \end{math}
    and $|x_0| \leq |x_i|$ for all $i$.
    There exists a one-variable equation $E'$ such that
    \begin{math}
        \sol(E') = \{[\eps], [x_0^{-1} x_1], \dots, [x_0^{-1} x_m]\}.
    \end{math}
\end{lemma}

\begin{proof}
    If $m = 0$, the claim is clear.
    Otherwise, we can assume that $E$ is of the form \eqref{eq:normalform}.
    Let $E'$ be the equation we get from $E$ by replacing $X$ by $x_0 X$:
    \begin{equation*}
        E': (u_0 x_0 X u_1 \dotsm x_0 X u_n, v_0 x_0 X v_1 \dotsm x_0 X v_n).
    \end{equation*}
    Because $E$ is nontrivial, $x_0$ is a prefix of every $x_i$.
    Clearly, the word $x_0^{-1} x_i$ is a solution word of $E'$.
    On the other hand, if $x$ is a solution word of $E'$,
    then $x_0 x$ is a solution word of $E$.
    This proves the claim.
\end{proof}

Next we will give an example of
how to transform an equation that satisfies Condition~N\ref{item:emptyzero}
into an equation in normal form.
After the example, we will prove that this can always be done.

\begin{example}
    Consider the equation
    \begin{equation*}
        (X ab X abab X a aba X b X, ab X X X abab a X a X bab).
    \end{equation*}
    By a length argument, it is equivalent to the system of equations
    \begin{equation*}
        (X ab, ab X), (X, X), (abab X, X abab), (a, a), (aba X b X, X a X bab).
    \end{equation*}
    We can drop the trivial equations $(X, X)$ and $(a, a)$,
    and then switch the left-hand and right-hand sides
    of the equations $(abab X, X abab)$ and $(aba X b X, X a X bab)$
    to get the system
    \begin{equation*}
        (X ab, ab X), (X abab, abab X), (X a X bab, aba X b X).
    \end{equation*}
    Then we can combine these equations into the equation
    \begin{equation*}
        (X ab X abab X a X bab, ab X abab X aba X b X),
    \end{equation*}
    which satisfies Conditions~N\ref{item:less} and N\ref{item:leq}.
    (Actually,
    this equation is equivalent to the equation $(X a X bab, aba X b X)$.)
\end{example}

\begin{lemma} \label{lem:normalform}
    Let $E$ be a nontrivial one-variable equation
    with the empty solution and at least one other solution.
    There exists an equation in normal form
    that is equivalent to $E$ up to a renaming of the letters
    and not longer than $E$.
\end{lemma}

\begin{proof}
    We can assume that
    $E$ has a nonempty zero-sum solution
    by Lemma~\ref{lem:alphabetnormalization}.
    We can also assume that
    $E$ is a shortest equation among all the equivalent equations,
    and $E$ is written as \eqref{eq:normalform}.
    Finally, we can let $j \in \{0, \dots, n\}$ be the smallest index such that
    $|u_0 \dotsm u_j| \geq |v_0 \dotsm v_j|$
    (the inequality holds for $j = n$, so $j$ exists),
    and assume that there does not exists an equivalent equally long equation
    for which the index $j$ would be larger.
    
    We are going to prove that $E$ is in normal form.
    We already know that Condition~N\ref{item:emptyzero} holds.
    
    If it were $j < n$ and $|u_0 \dotsm u_j| = |v_0 \dotsm v_j|$,
    then for any word $x$ we would have the sequence of equivalences
    \begin{align*}
        & u_0 x u_1 \dotsm x u_n = v_0 x v_1 \dotsm x v_n \\
        \iff & u_0 x u_1 \dotsm x u_j = v_0 x v_1 \dotsm x v_j
            \land u_{j + 1} x u_{j + 2} \dotsm x u_n
                = v_{j + 1} x v_{j + 2} \dotsm x v_n \\
        \iff & u_0 x u_1 \dotsm x u_j u_{j + 1} x u_{j + 2} \dotsm x u_n
            = v_0 x v_1 \dotsm x v_j v_{j + 1} x v_{j + 2} \dotsm x v_n,
    \end{align*}
    so $E$ would be equivalent to the shorter equation
	\begin{equation*}
	    (u_0 X u_1 \dotsm X u_j u_{j + 1} X u_{j + 2} \dotsm X u_n,
	    v_0 X v_1 \dotsm X v_j v_{j + 1} X v_{j + 2} \dotsm X v_n),
	\end{equation*}
	which would contradict the minimality of $E$.
    On the other hand,
    if it were $j < n$ and $|u_0 \dotsm u_j| > |v_0 \dotsm v_j|$,
    then there would exist words $p, q$ such that
    $u_j = pq$ and $|u_0 \dotsm u_{j - 1} p| = |v_0 \dotsm v_j|$,
    and for any word $x$ we would have the sequence of equivalences
    \begin{align*}
        & u_0 x u_1 \dotsm x u_n = v_0 x v_1 \dotsm x v_n \\
        \iff & u_0 x u_1 \dotsm x u_{j - 1} x p = v_0 x v_1 \dotsm x v_j
            \land q x u_{j + 1} \dotsm x u_n = x v_{j + 1} \dotsm x v_n \\
        \iff & u_0 x u_1 \dotsm x u_{j - 1} x p x v_{j + 1} \dotsm x v_n
            = v_0 x v_1 \dotsm x v_j q x u_{j + 1} \dotsm x u_n,
    \end{align*}
    so $E$ would be equivalent to the equation
	\begin{equation*}
	    ((u_0 X u_1 \dotsm X u_{j - 1} X p X v_{j + 1} \dotsm X v_n,
        v_0 X v_1 \dotsm X v_j q X u_{j + 1} \dotsm X u_n),
	\end{equation*}
	which would contradict the minimality of $j$.
    The only possibility is that $j = n$,
	so Condition~N\ref{item:less} holds.
    
    If there were an index $i \in \{0, \dots, n\}$ such that
    $|u_0 \dotsm u_i| > |v_0 \dotsm v_{i - 1}|$,
    then there would exist words $p, q, r$ such that $u_i = pq$, $v_i = qr$,
    and $|u_0 \dotsm u_{i - 1} p| = |v_0 \dotsm v_{i - 1}|$,
    and for any word $x$ we would have the sequence of equivalences
    \begin{align*}
        & u_0 x u_1 \dotsm x u_n = v_0 x v_1 \dotsm x v_n \\
        \iff & u_0 x u_1 \dotsm x u_{i - 1} x p = v_0 x v_1 \dotsm x v_{i - 1} x
            \land x u_{i + 1} \dotsm x u_n = r x v_{i + 1} \dotsm x v_n \\
        \iff & u_0 x u_1 \dotsm x u_{i - 1} x p x u_{i + 1} \dotsm x u_n
            = v_0 x v_1 \dotsm x v_{i - 1} x r x v_{i + 1} \dotsm x v_n,
    \end{align*}
    so $E$ would be equivalent to the shorter equation
	\begin{equation*}
	    (u_0 X u_1 \dotsm X u_{i - 1} X p X u_{i + 1} \dotsm X u_n
        = v_0 X v_1 \dotsm X v_{i - 1} X r X v_{i + 1} \dotsm X v_n),
	\end{equation*}
    which would contradict the minimality of $E$.
	This shows that also Condition~N\ref{item:leq} holds,
	so $E$ is in normal form.
\end{proof}

\section{Sums and heights of solutions}

In this section, we prove lemmas about the sums and heights of solution words
of one-variable equations in normal form.

\begin{lemma} \label{lem:zerosum}
    All solutions of an equation in normal form are zero-sum.
\end{lemma}

\begin{proof}
    Let the equation be \eqref{eq:normalform}.
    Let $u_i' = u_0 \dotsm u_{i - 1}$
    and $v_i' = v_0 \dotsm v_{i - 1}$ for all $i$.
    After applying a solution $[x]$ on the left-hand side
    and taking the area we get
    \begin{align*}
        &\area{u_0 x u_1 \dotsm x u_n} \\
        = &\sum_{i = 0}^n
            (\area{u_i} + \wordsum(u_0 x u_1 \dotsm u_{i - 1} x) |u_i|) +
            \sum_{i = 1}^n
            (\area{x} + \wordsum(u_0 x u_1 \dotsm x u_{i - 1}) |x|) \\
        = &\sum_{i = 0}^n
            (\area{u_i} + \wordsum(u_i') |u_i| + i \wordsum(x) |u_i|) +
            \sum_{i = 1}^n
            (\area{x} + \wordsum(u_i') |x| + (i - 1) \wordsum(x) |x|) \\
        = &\area{u_0 \dotsm u_n} + \wordsum(x) \sum_{i = 0}^n i |u_i| +
            n \area{x} + |x| \sum_{i = 1}^n \wordsum(u_i') +
            \frac{(n - 1) n}{2} \cdot \wordsum(x) |x|.
    \end{align*}
    We get a similar formula for $\area{v_0 x v_1 \dotsm x v_n}$.
    Because $u_0 x u_1 \dotsm x u_n = v_0 x v_1 \dotsm x v_n$, we get
    \begin{equation} \label{eq:areadiff}
    \begin{split}
        0 &= \area{u_0 x u_1 \dotsm x u_n} - \area{v_0 x v_1 \dotsm x v_n} \\
        &= \area{u_0 \dotsm u_n} - \area{v_0 \dotsm v_n} +
            \wordsum(x) \sum_{i = 0}^n i (|u_i| - |v_i|) +
            |x| \sum_{i = 1}^n (\wordsum(u_i') - \wordsum(v_i')) \\
        &= \wordsum(x) \sum_{i = 0}^n i (|u_i| - |v_i|) +
            |x| \sum_{i = 1}^n (\wordsum(u_i') - \wordsum(v_i')).
    \end{split}
    \end{equation}
    By the definition of normal form,
    the equation has a nonempty zero-sum solution $[x_1]$.
    Replacing $x$ by $x_1$ in \eqref{eq:areadiff} gives
    \begin{equation*}
        0 = |x_1| \sum_{i = 1}^n (\wordsum(u_i') - \wordsum(v_i')).
    \end{equation*}
    Because $|x_1| > 0$,
    \begin{math}
        \sum_{i = 1}^n (\wordsum(u_i') - \wordsum(v_i')) = 0.
    \end{math}
    Then \eqref{eq:areadiff} takes the form
    \begin{equation*}
        0 = \wordsum(x) \sum_{i = 0}^n i (|u_i| - |v_i|),
    \end{equation*}
    so either $\wordsum(x) = 0$ or
    \begin{math}
        \sum_{i = 0}^n i (|u_i| - |v_i|) = 0.
    \end{math}
    The latter is not possible, because
    \begin{align*}
        &\sum_{i = 0}^n i (|u_i| - |v_i|)
            = \sum_{i = 1}^n (|u_i \dotsm u_n| - |v_i \dotsm v_n|) \\
        = &\sum_{i = 1}^n
            (|u_0 \dotsm u_n| - |u_i'| - (|v_0 \dotsm v_n| - |v_i'|))
            = \sum_{i = 1}^n (-|u_i'| + |v_i'|) > 0,
    \end{align*}
    by Condition~N\ref{item:less} in the definition of normal form.
    Thus every solution $[x]$ is zero-sum.
\end{proof}

\begin{lemma} \label{lem:st}
    Consider the nontrivial equation \eqref{eq:normalform}.
    Let $s_i = \wordsum(u_0 \dotsm u_{i - 1})$
    and $t_i = \wordsum(v_0 \dotsm v_{i - 1})$ for all $i$.
    If the equation has at least two zero-sum solutions,
    then $(s_1, \dots, s_n)$ is a permutation of $(t_1, \dots, t_n)$.
\end{lemma}

\begin{proof}
    Let $[x]$ and $[y]$ be two zero-sum solutions and let $|x| > |y|$.
    Because $y$ is a prefix and a suffix of $x$,
    also $\psw_r(y)$ is a prefix and a suffix of $\psw_r(x)$ for every $r$.
    Consequently, every letter that appears in $\psw_r(y)$
    appears more often in $\psw_r(x)$.
    Let $(s_1', \dots, s_n')$ be the permutation of $(s_1, \dots, s_n)$
    such that $s_i' \leq s_{i + 1}'$ for all $i$,
    and let $(t_1', \dots, t_n')$ be the permutation of $(t_1, \dots, t_n)$
    such that $t_i' \leq t_{i + 1}'$ for all $i$.
    Let $j$ be the largest index such that $s_j' \ne t_j'$
    (if there is no such index, then we have proved the lemma).
    Without loss of generality, let $s_j' > t_j'$.
    Let $a = \height{x} + s_j'$.
    If the number of occurrences of $a$ in any word $w$ is denoted by $|w|_a$,
    then
    \begin{align}
        0 = &|\psw(u_0 x u_1 \dotsm x u_n)|_a
            - |\psw(v_0 x v_1 \dotsm x v_n)|_a \nonumber \\
            &- |\psw(u_0 y u_1 \dotsm y u_n)|_a
            + |\psw(v_0 y v_1 \dotsm y v_n)|_a \label{eq:st1} \\
        = &\sum_{i = 1}^n (|\psw_{s_i'}(x)|_a - |\psw_{t_i'}(x)|_a
            - |\psw_{s_i'}(y)|_a + |\psw_{t_i'}(y)|_a) \label{eq:st2} \\
        = &\sum_{i = 1}^j (|\psw_{s_i'}(x)|_a - |\psw_{t_i'}(x)|_a
            - |\psw_{s_i'}(y)|_a + |\psw_{t_i'}(y)|_a) \label{eq:st3} \\
        = &\sum_{i = 1}^j (|\psw_{s_i'}(x)|_a - |\psw_{s_i'}(y)|_a)
            \label{eq:st4} \\
        \geq &|\psw_{s_j'}(x)|_a - |\psw_{s_j'}(y)|_a > 0, \label{eq:st5}
    \end{align}
    a contradiction.
    Here, \eqref{eq:st1} follows from $x$ and $y$ being solution words,
    \eqref{eq:st2} from them being zero-sum,
    \eqref{eq:st3} from the definition of $j$,
    \eqref{eq:st4} from
    \begin{math}
        a > \height{x} + t_j' \geq \height{x} + t_i' \geq \height{y} + t_i'
    \end{math}
    for all $i \in \{1, \dots, j\}$,
    and \eqref{eq:st5} from $|\psw_{s_j'}(x)|_a > 0$
    and the fact that for all $r$, every letter that appears in $\psw_r(y)$
    appears more often in $\psw_r(x)$.
\end{proof}

\begin{lemma} \label{lem:height}
    Let \eqref{eq:normalform} be an equation in normal form.
    Let
    \begin{equation} \label{eq:height}
        h = \height{u_0 \dotsm u_n}
        - \max\{\wordsum(u_0 \dotsm u_i) \mid i \in \{0, \dots, n - 1\}\}.
    \end{equation}
    If the equation has at least three nonempty solutions,
    then every nonempty solution is of height $h$.
    If the equation has two nonempty solutions,
    then the shorter one is of height $h$
    and the longer one of height at least $h$.
\end{lemma}

\begin{proof}
    The idea of the proof
    is to look at the first occurrences of the highest points on the curves
    of the left-hand side and the right-hand side of the equation;
    these must match.
    If the length of the solution changes,
    these first occurrences often move with respect to each other
    so that they no longer match;
    this puts a limit on the number of solutions under certain conditions.
    A first occurrence can be either inside a constant part
    or inside a variable.
    We will see that
    if the first occurrences are inside constant parts on both sides,
    then the solution is empty,
    if they are inside variables on both sides,
    then the solution is of height at least $h$
    and there can be at most one solution of height more than $h$,
    and if the first occurrence is inside a constant part on one side
    and inside a variable on the other side,
    then the solution is of height $h$,
    and if there is a solution of height more than $h$,
    then there can be at most one solution of height $h$.

    For any word $w$, let $\phi(w)$ be its shortest prefix such that
    \begin{math}
        \height{\phi(w)} = \height{w}.
    \end{math}
    For any solution $[x]$, we have
    \begin{equation} \label{eq:phi}
        \phi(u_0 x u_1 \dotsm x u_n) = \phi(v_0 x v_1 \dotsm x v_n).
    \end{equation}
    Let $s_i = \wordsum(u_0 \dotsm u_{i - 1})$
    and $t_i = \wordsum(v_0 \dotsm v_{i - 1})$ for all $i$.
    Let $i$ and $j$ be such that
    $\phi(u_0 \dotsm u_n) = u_0 \dotsm u_{i - 1} \phi(u_i)$ and
    $\phi(v_0 \dotsm v_n) = v_0 \dotsm v_{j - 1} \phi(v_j)$.
    Because $[\eps]$ is a solution,
    $\phi(u_0 \dotsm u_n) = \phi(v_0 \dotsm v_n)$ and thus
    \begin{equation} \label{eq:lowmatch}
        |u_0 \dotsm u_{i - 1}| + |\phi(u_i)|
        = |v_0 \dotsm v_{j - 1}| + |\phi(v_j)|.
    \end{equation}
    By \eqref{eq:lowmatch}
    and Condition~N\ref{item:leq} in the definition of normal form, $i > j$.

    Because $[\eps]$ is a solution,
    \begin{math}
        \height{u_0 \dotsm u_n} = \height{v_0 \dotsm v_n},
    \end{math}
    and by Lemma~\ref{lem:st},
    \begin{equation*}
        \max\{\wordsum(u_0 \dotsm u_i) \mid i \in \{0, \dots, n - 1\}\}
        = \max\{\wordsum(v_0 \dotsm v_i) \mid i \in \{0, \dots, n - 1\}\},
    \end{equation*}
    so
    \begin{equation*}
        h = \height{v_0 \dotsm v_n}
        - \max\{\wordsum(v_0 \dotsm v_i) \mid i \in \{0, \dots, n - 1\}\}.
    \end{equation*}
    Let $k$ and $l$ be the smallest indices such that
    $s_k = \max\{s_1, \dots, s_n\}$ and
    $t_l = \max\{t_1, \dots, t_n\}$.
    Then
    \begin{align*}
        \phi(u_0 x u_1 \dotsm x u_n) &= \begin{cases}
            u_0 x u_1 \dotsm u_{i - 1} x \phi(u_i)
            &\text{if $\height{x} < h$
                or if $\height{x} = h$ and $i < k$}, \\
            u_0 x u_1 \dotsm x u_{k - 1} \phi(x)
            &\text{if $\height{x} > h$
                or if $\height{x} = h$ and $i \geq k$},
        \end{cases} \\
        \phi(v_0 x v_1 \dotsm x v_n) &= \begin{cases}
            v_0 x v_1 \dotsm v_{j - 1} x \phi(v_j)
            &\text{if $\height{x} < h$
                or if $\height{x} = h$ and $j < l$}, \\
            v_0 x v_1 \dotsm x v_{l - 1} \phi(x)
            &\text{if $\height{x} > h$
                or if $\height{x} = h$ and $j \geq l$},
        \end{cases}
    \end{align*}
    This means that, for a given $x$,
    \eqref{eq:phi} can take one of four possible forms:
    \begin{enumerate}[(i)]
        \item
        If $\height{x} < h$ or if $\height{x} = h$, $i < k$ and $j < l$, then
        \begin{equation*}
            u_0 x u_1 \dotsm u_{i - 1} x \phi(u_i)
            = v_0 x v_1 \dotsm v_{j - 1} x \phi(v_j)
        \end{equation*}
        and thus
        \begin{equation*}
            |u_0 \dotsm u_{i - 1}| + |\phi(u_i)| + (i - j) |x|
            = |v_0 \dotsm v_{j - 1}| + |\phi(v_j)|.
        \end{equation*}
        Because $i > j$,
        it follows that this equality can hold for at most one $|x|$,
        so there is only one possible $x$ in this case, namely, the empty word.
        \item
        If $\height{x} = h$, $i < k$ and $j \geq l$, then
        \begin{equation*}
            u_0 x u_1 \dotsm u_{i - 1} x \phi(u_i)
            = v_0 x v_1 \dotsm x v_{l - 1} \phi(x),
        \end{equation*}
        but
        \begin{align*}
            &|u_0 x u_1 \dotsm u_{i - 1} x \phi(u_i)|
            = |u_0 \dotsm u_{i - 1}| + |\phi(u_i)| + i|x|
            = |v_0 \dotsm v_{j - 1}| + |\phi(v_j)| + i|x| \\
            >& |v_0 \dotsm v_{l - 1}| + l|x|
            \geq |v_0 x v_1 \dotsm x v_{l - 1} \phi(x)|
        \end{align*}
        by \eqref{eq:lowmatch} and $i > j \geq l$, a contradiction.
        \item \label{item:3}
        If $\height{x} > h$ or if $\height{x} = h$, $i \geq k$ and $j \geq l$,
        then
        \begin{equation*}
            u_0 x u_1 \dotsm x u_{k - 1} \phi(x)
            = v_0 x v_1 \dotsm x v_{l - 1} \phi(x)
        \end{equation*}
        and thus
        \begin{equation*}
            |u_0 \dotsm u_{k - 1}| + (k - l) |x|
            = |v_0 \dotsm v_{l - 1}|.
        \end{equation*}
        By Condition~N\ref{item:less} in the definition of normal form, $k > l$.
        It follows that this equality can hold for at most one $|x|$,
        so there is only one possible $x$ in this case.
        \item
        If $\height{x} = h$, $i \geq k$ and $j < l$, then
        \begin{equation*}
            u_0 x u_1 \dotsm x u_{k - 1} \phi(x)
            = v_0 x v_1 \dotsm v_{j - 1} x \phi(v_j)
        \end{equation*}
        and thus
        \begin{equation} \label{eq:item4}
            |u_0 \dotsm u_{k - 1}| + |\phi(x)| + (k - 1 - j) |x|
            = |v_0 \dotsm v_{j - 1}| + |\phi(v_j)|.
        \end{equation}
        If $x$ and $x'$ are solution words,
        then one of them is a prefix of the other,
        so if they have the same height, then $\phi(x) = \phi(x')$.
        Therefore, \eqref{eq:item4} can hold for more than one
        solution word $x$ of height $h$ only if $k - 1 - j = 0$.
        In general, this can happen
        (for example, if the equation has infinitely many solutions).
        However, if there exists a solution word of height more than $h$,
        then it follows from Case \eqref{item:3} that $k > l$.
        Then $j < l < k$, so $k - 1 > j$
        and there is at most one solution word $x$ of height $h$.
        \qedhere
    \end{enumerate}
\end{proof}

\begin{example}
    Consider the equation
    \begin{equation*}
        (X aXb X aabb ab aXb abaabbab, abaabbab aXb ab aabb X aXb X)
    \end{equation*}
    that was mentioned in Example~\ref{exa:equations2}.
    Let $a = 1$ and $b = -1$.
    The equation has exactly three solutions $[\eps], [ab], [abaabbab]$.
    All of them are zero-sum,
    and their heights are $-\infty$, 1, 2, respectively.
    If we use the notation of the proof of Lemma~\ref{lem:height},
    then $i = 3$, $j = 0$, $k = 2$, $l = 1$, and $h = 1$.
    We have $\phi(u_i) = \phi(aabb ab a) = aa$,
    $\phi(v_j) = \phi(abaabbab a) = abaa$,
    $\phi(ab) = a$, and $\phi(abaabbab) = abaa$.
    Then
    \begin{align*}
        \phi(x axb x aabb ab axb abaabbab) &= \begin{cases}
            x axb x aa
            &\text{if $x = \eps$}, \\
            x a \phi(x)
            &\text{if $x = abaabbab$ or if $x = ab$},
        \end{cases} \\
        \phi(abaabbab axb ab aabb x axb x) &= \begin{cases}
            abaa
            &\text{if $x = \eps$ or if $x = ab$}, \\
            abaabbab a \phi(x)
            &\text{if $x = abaabbab$}.
        \end{cases}
    \end{align*}
\end{example}

\section{Some Lemmas}

In this section, we state many lemmas about one-variable equations
that will be used in the proof of the main result.

A subset $Z$ of $\Gamma^*$ is called a \emph{code}
if the elements of $Z$ do not satisfy any nontrivial relations.
In other words, $Z$ is a code if and only if
for all $x_1, \dots, x_m, y_1, \dots, y_n \in Z$,
$x_1 \dotsm x_m = y_1 \dotsm y_n$ implies
$m = n$ and $x_i = y_i$ for all $i \in \{1, \dots, m\}$.
If $Z$ is a code, then $Z^*$ is a free monoid,
and if $\Delta$ is an alphabet of the same size as $Z$,
then the free monoids $Z^*$ and $\Delta^*$ are isomorphic.
More information about codes can be found
in the book of Berstel, Perrin and Reutenauer~\cite{BePeRe10}.

The next lemma can be used to compress an equation into a shorter one.
We will use it with two codes $Z$:
The set of all minimal zero-sum words
(those zero-sum words
which cannot be written as a product of two shorter zero-sum words),
and the set of words of a specific length.

\begin{lemma} \label{lem:code}
    Let $E$ be the equation \eqref{eq:normalform} and let $Z$ be a code.
    If $u_i, v_i \in Z^*$ for all $i$,
    then there exists an alphabet $\Delta$
    and an isomorphism $h: Z^* \to \Delta^*$,
    and the equation
    \begin{equation} \label{eq:compressedeq}
        (h(u_0) X h(u_1) \dotsm X h(u_n), h(v_0) X h(v_1) \dotsm X h(v_n))
    \end{equation}
    has the solution set
    \begin{math}
        \{[h(x)] \mid [x] \in \sol(E), \ x \in Z^*\}.
    \end{math}
\end{lemma}

\begin{proof}
    There exists an alphabet $\Delta$
    and an isomorphism $h: Z^* \to \Delta^*$ by the definition of code.
    If $x \in Z^*$ is a solution word of $E$, then
    \begin{align*}
        h(u_0) h(x) h(u_1) \dotsm h(x) h(u_n)
        &= h(u_0 x u_1 \dotsm x u_n) \\
        &= h(v_0 x v_1 \dotsm x v_n)
        = h(v_0) h(x) h(v_1) \dotsm h(x) h(v_n),
    \end{align*}
    so $[h(x)]$ is a solution of \eqref{eq:compressedeq}.
    On the other hand, if $[y]$ is a solution of \eqref{eq:compressedeq},
    then there exists $x \in Z^*$ such that $h(x) = y$, and
    \begin{align*}
        h(u_0 x u_1 \dotsm x u_n)
        &= h(u_0) y h(u_1) \dotsm y h(u_n) \\
        &= h(v_0) y h(v_1) \dotsm y h(v_n)
        = h(v_0 x v_1 \dotsm x v_n),
    \end{align*}
    so $u_0 x u_1 \dotsm x u_n = v_0 x v_1 \dotsm x v_n$
    and $[x]$ is a solution of $E$.
    This completes the proof.
\end{proof}

Note that the equation $E$ in Lemma~\ref{lem:code}
can have solution words that are not in $Z^*$,
so \eqref{eq:compressedeq} can have less solutions than $E$.

The next lemma can be used to cut off part of an equation
so that all solutions are preserved, except possibly the empty solution
(and maybe some additional solutions are added).

\begin{lemma} \label{lem:cut}
    Consider the equation \eqref{eq:normalform}.
    Let $k \in \{0, \dots, n\}$ and let
    \begin{equation*}
        d = |v_0 \dotsm v_{k - 1}| - |u_0 \dotsm u_k| \geq 0.
    \end{equation*}
    If all nonempty solutions of the equation are of length at least $d$,
    and if $y$ is the common prefix of length $d$
    of all nonempty solution words,
    then each one of the nonempty solutions is a solution of the equation
    \begin{equation} \label{eq:cut}
        (u_0 X u_1 \dotsm X u_k y, v_0 X v_1 \dotsm v_{k - 1} X).
    \end{equation}
\end{lemma}

\begin{proof}
    If $h$ is a nonempty solution of \eqref{eq:normalform}, then
    \begin{equation*}
        h(u_0 X u_1 \dotsm X u_n) = h(v_0 X v_1 \dotsm X v_n).
    \end{equation*}
    Here the left-hand side has a prefix
    \begin{math}
        h(u_0 X u_1 \dotsm X u_k y)
    \end{math}
    and the right-hand side has a prefix
    \begin{math}
        h(v_0 X v_1 \dotsm v_{k - 1} X).
    \end{math}
    These prefixes are of the same length, so they are equal.
    Thus $h$ is a solution of \eqref{eq:cut}.
\end{proof}


Using Lemma \ref{lem:cut} requires the existence of a suitable index $k$.
The next two lemmas can sometimes be used to find such an index.
The proof of Lemma~\ref{lem:cutpoint}
is somewhat similar to the proof of Lemma \ref{lem:height}, but simpler.

\begin{lemma} \label{lem:cutpoint}
    Let \eqref{eq:normalform} be an equation in normal form.
    If it has at least three nonempty solutions,
    and if there exists $k \in \{1, \dots, n - 1\}$ such that
    \begin{equation*}
        \wordsum(u_0) = \dots = \wordsum(u_{k - 1}) = 0 \ne \wordsum(u_k),
    \end{equation*}
    then every nonempty solution is of length more than
    \begin{math}
        |v_0 \dotsm v_{k - 1}| - |u_0 \dotsm u_k|.
    \end{math}
\end{lemma}

\begin{proof}
    By symmetry, we can assume that $\wordsum(u_k) > 0$.
    By Lemma \ref{lem:height}, the nonempty solutions have a common height $h$.
    For any word $w$ of height at least $\wordsum(u_k) + h$,
    let $\psi(w)$ be its shortest prefix such that
    \begin{math}
        \height{\psi(w)} \geq \wordsum(u_k) + h.
    \end{math}
    If $[x]$ is a nonempty solution,
    then there exist indices $i, j$ and words $u, v$ such that
    $u$ is a nonempty prefix of $u_i x$,
    $v$ is a nonempty prefix of $v_j x$ and
    \begin{equation*}
        \psi(u_0 x u_1 \dotsm x u_n) = u_0 x u_1 \dotsm u_{i - 1} x u, \
        \psi(v_0 x v_1 \dotsm x v_n) = v_0 x v_1 \dotsm v_{j - 1} x v.
    \end{equation*}
    Here $i, j, u, v$ are the same for all $x$,
    because every $x$ has sum zero and height $h$,
    and the shortest $x$ is a prefix of every other $x$.
    Clearly $i \leq k$, because
    \begin{equation*}
        \height{u_0 x u_1 \dotsm u_k x}
        \geq \wordsum(u_0 x u_1 \dotsm x u_k) + h = \wordsum(u_k) + h.
    \end{equation*}
    We know that
    \begin{math}
        \psi(u_0 x u_1 \dotsm x u_n) = \psi(v_0 x v_1 \dotsm x v_n)
    \end{math}
    (actually, we only need the fact that these words have the same length).
    Because
    \begin{equation*}
        |u_0 x u_1 \dotsm u_{i - 1} x u|
        = |v_0 x v_1 \dotsm v_{j - 1} x v|
    \end{equation*}
    for more than one $|x|$, it must be $i = j$,
    and then $|u_0 \dotsm u_{i - 1} u| = |v_0 \dotsm v_{i - 1} v|$.
    Because $|u_0 \dotsm u_i| \leq |v_0 \dotsm v_{i - 1}|$
    by Condition~N\ref{item:leq} in the definition of normal form,
    $u$ cannot be a prefix of $u_i$.
    This means that $\height{u_0 x u_1 \dotsm x u_i} < \wordsum(u_k) + h$.
    If $i < k$, then $u_i$ is zero-sum
    and thus adding $x$ after $x u_i$ does not increase the height,
    so also $\height{u_0 x u_1 \dotsm u_i x} < \wordsum(u_k) + h$,
    which is a contradiction.
    Therefore $i = k$.
    If there exists a nonempty solution $[x]$ of length at most
    \begin{math}
        |v_0 \dotsm v_{k - 1}| - |u_0 \dotsm u_k|,
    \end{math}
    then
    \begin{equation*}
        |u_0 \dotsm u_{k - 1} u| \leq |u_0 \dotsm u_k x|
        \leq |v_0 \dotsm v_{k - 1}| < |v_0 \dotsm v_{k - 1} v|,
    \end{equation*}
    a contradiction.
\end{proof}

\begin{lemma} \label{lem:cutpointper}
    Let the equation \eqref{eq:normalform} have the solution set $[p^*]$
    for some primitive word $p$.
    Let $u_0 = v_n = \eps$.
    Let $j \in \{0, \dots, n\}$ be the largest index such that
    the lengths of $u_0, \dots, u_{j - 1}$ and $v_0, \dots, v_{j - 1}$
    are divisible by $|p|$.
    Then $j > 0$ and
    \begin{math}
        |v_0 \dotsm v_{j - 1}| - |u_0 \dotsm u_j| \leq |p|.
    \end{math}
\end{lemma}

\begin{proof}
    If $j = n$, the claim is clear.
    Otherwise, at least one of $|u_j|, |v_j|$ is not divisible by $|p|$.
    Let $m$ be such that
    \begin{math}
         |p^{m - 1}| \geq |v_0 \dotsm v_j| - |u_0 \dotsm u_j|.
    \end{math}
    Let
    \begin{math}
        d = |v_0 \dotsm v_{j - 1}| - |u_0 \dotsm u_j|.
    \end{math}

    Let $r$ be the prefix of $p^m$ of length
    \begin{math}
         |p^m| - |v_0 \dotsm v_j| + |u_0 \dotsm u_j| \geq |p|,
    \end{math}
    and let $p'$ be the suffix of $r$ of length $|p|$.
    Because $p$ is primitive,
    $p' = p$ if and only if $|r|$ is divisible by $|p|$.
    We have
    \begin{math}
        u_0 p^m u_1 \dotsm u_j p^m = v_0 p^m v_1 \dotsm p^m v_j r,
    \end{math}
    and it follows that $p = p'$, so $|r|$ is divisible by $|p|$.
    This means that $|u_j|$ and $|v_j|$ are congruent modulo $|p|$,
    so neither of them is divisible by $|p|$.
    Consequently, $j \ne 0$ and $d$ is not divisible by $|p|$.

    Let $s$ be the prefix of $p^m$ of length $d$.
    If $d > |p|$, we can let $p''$ be the suffix of $s$ of length $|p|$.
    Because $p$ is primitive,
    $p'' = p$ if and only if $|s|$ is divisible by $|p|$.
    We have
    \begin{math}
        u_0 p^m u_1 \dotsm p^m u_j s = v_0 p^m v_1 \dotsm v_{j - 1} p^m,
    \end{math}
    and it follows that $p = p''$, so $|s| = d$ is divisible by $|p|$.
    This is a contradiction, so $d \leq |p|$.
\end{proof}


Lemma \ref{lem:cut} does not guarantee
that the new, shorter equation would have the empty solution.
Sometimes the next lemma can be used to get around this problem.

\begin{lemma} \label{lem:empty}
    If the equation \eqref{eq:normalform} has a nonempty solution,
    $u_n = u a^m$ for some $u \in \Gamma^*$, $a \in \Gamma$ and $m \geq 0$,
    and $u_0 \dotsm u_{n - 1} u$ is a prefix of $v_0 \dotsm v_n$,
    then the equation has the empty solution.
\end{lemma}

\begin{proof}
    Let $y$ be a word such that
    \begin{math}
        u_0 \dotsm u_{n - 1} u y = v_0 \dotsm v_n.
    \end{math}
    We say that words $p, q$ are \emph{abelian equivalent}
    if $|p|_b = |q|_b$ for all letters $b$.
    Because \eqref{eq:normalform} has a solution,
    $u_0 \dotsm u_{n - 1} u a^m$ and $v_0 \dotsm v_n$ are abelian equivalent.
    Thus $u_0 \dotsm u_{n - 1} u y$ and $u_0 \dotsm u_{n - 1} u a^m$
    are abelian equivalent,
    so $y$ and $a^m$ are abelian equivalent and $y = a^m$.
    The claim follows.
\end{proof}

\section{Main results}

Now we are ready to prove our main results.

\begin{theorem} \label{thm:main}
    If a one-variable equation has only finitely many solutions,
    it has at most three solutions.
\end{theorem}

\begin{proof}
    Assume that there is a counterexample.
    Then there is one with an empty solution by Lemma~\ref{lem:assumeempty}.
    Of all equations with the empty solution,
    at least three nonempty solutions,
    and only finitely many solutions,
    let $E_1$ be a shortest one.
    We are going to prove a contradiction by showing that
    there exists a shorter equation with these properties.
    By Lemma~\ref{lem:normalform},
    we can assume that $E_1$ is the equation \eqref{eq:normalform}
    and it is in normal form.
    By Lemma~\ref{lem:zerosum}, each one of its solutions is zero-sum.

    The idea of the proof is to cut off part of the equation
    to get a shorter equation $E_2$
    that has at least three nonempty solutions but only finitely many.
    Unfortunately, $E_2$ does not necessarily have the empty solution.
    We map $E_2$ with a length-preserving mapping to get an equation $E_3$
    that has at least three nonempty solution and also the empty solution.
    Unfortunately, $E_3$ might have infinitely many solutions.
    We analyze $E_3$ to find another way to cut off part of $E_1$
    to get an equation $E_4$, which is then modified to an equation $E_5$.
    For $E_5$, we can finally prove that it has the empty solution
    and at least three but only finitely many nonempty solutions.

    If $\wordsum(u_i) = 0$ for all $i < n$,
    then $\wordsum(v_i) = 0$ for all $i < n$ by Lemma~\ref{lem:st},
    and then also $\wordsum(u_n) = 0$,
    because $\wordsum(u_0 \dotsm u_n) = \wordsum(v_0 \dotsm v_n)$
    and $v_n = \eps$.
    Thus all $u_i, v_i$ are zero-sum,
    and we can use Lemma~\ref{lem:code}
    with $Z$ the set of all minimal zero-sum words
    to get a shorter equation
    with the same number of solutions, one of them empty.

    For the rest of the proof, we assume that there exists a minimal $k < n$
    such that $\wordsum(u_k) \ne 0$.
    By symmetry, we can assume that $\wordsum(u_k) > 0$.
    By Lemmas~\ref{lem:cutpoint} and \ref{lem:cut},
    we get a shorter equation
    \begin{equation*}
        E_2: (u_0 X u_1 \dotsm X u_k y, v_0 X v_1 \dotsm v_{k - 1} X)
    \end{equation*}
    that has at least all the same nonempty solutions as $E_1$.
    It might have some other solutions as well,
    but it cannot have infinitely many solutions,
    because the intersection of
    an infinite solution set of a nontrivial one-variable equation
    and a finite solution set of a one-variable equation
    is of size at most two
    by Theorem~\ref{thm:one_sol_log} and Lemma~\ref{lem:one_pq}.
    If it has also the empty solution, then we are done,
    but we do not know yet whether this is the case.
    We can use Lemma~\ref{lem:st} for $E_2$ to see that
    \begin{math}
        (\wordsum(u_0), \dots, \wordsum(u_0 \dotsm u_{k - 1}))
    \end{math}
    and
    \begin{math}
        (\wordsum(v_0), \dots, \wordsum(v_0 \dotsm v_{k - 1}))
    \end{math}
    are permutations of each other.
    We know that $u_0, \dots, u_{k - 1}$ are zero-sum,
    so also $v_0, \dots, v_{k - 1}$ are zero-sum.

    Let $[x_1]$ be the shortest nonempty solution of $E_1$.
    Let $\{a, b\}$ be an alphabet
    and let $g$ be the morphism that maps
    the letter $\min(\psw(x_1))$ to $b$ and every other letter to $a$.
    Let $f = g \circ \psw$.
    Then $f$ is length-preserving,
    and if $w$ is zero-sum, then $f(w w') = f(w) f(w')$.
    If $[x]$ is a nonempty solution of $E_1$,
    then $[f(x)]$ is a solution of the equation
    \begin{equation*}
        E_3: (f(u_0) X f(u_1) \dotsm X f(u_k y),
        f(v_0) X f(v_1) \dotsm f(v_{k - 1}) X).
    \end{equation*}
    We have
    \begin{math}
        f(u_k y) = f(u_k) g(\psw_{\wordsum(u_k)}(y)).
    \end{math}
    Because $\wordsum(u_k) > 0$ and $y$ is a prefix of $x_1$,
    \begin{math}
        \min(\psw_{\wordsum(u_k)}(y)) > \min(\psw(x_1)).
    \end{math}
    Thus
    \begin{math}
        g(\psw_{\wordsum(u_k)}(y)) \in a^*.
    \end{math}
    Because $u_0 \dotsm u_k$ is a prefix of $v_0 \dotsm v_{k - 1}$,
    also $f(u_0 \dotsm u_k) = f(u_0) \dotsm f(u_k)$
    is a prefix of $f(v_0 \dotsm v_{k - 1}) = f(v_0) \dotsm f(v_{k - 1})$.
    We can use Lemma~\ref{lem:empty}
    with $g(\psw_{\wordsum(u_k)}(y))$ as $a^m$,
    so $E_3$ has the empty solution.
    If it has only finitely many solutions, then we are done.
    For the rest of the proof, we assume that it has infinitely many solutions.
    Then its solution set is $[p^*]$ for some primitive word $p$.
    Consequently,
    the length of every solution word of $E_1$ is divisible by $|p|$.
    Because the solution word $f(x_1)$ of $E_3$ contains the letter $b$,
    also $p$ must contain $b$.
    This means that $p$ cannot be a suffix of
    \begin{math}
        g(\psw_{\wordsum(u_k)}(y)) \in a^*,
    \end{math}
    so $|p| > |y|$.

    We can use Lemma~\ref{lem:cutpointper} for $E_3$ to find an index $j$
    such that the lengths of
    $u_0, \dots, u_{j - 1}$ and $v_0, \dots, v_{j - 1}$
    are divisible by $|p|$ and, if $j < k$,
    \begin{math}
        |v_0 \dotsm v_{j - 1}| - |u_0 \dotsm u_j| \leq |p|
    \end{math}
    (remember that $f$ is length-preserving).
    By letting $z = y$ if $j = k$,
    or by using Lemma~\ref{lem:cut} with $j$ as $k$ for $E_1$ otherwise,
    we get an equation
    \begin{equation*}
        E_4: (u_0 X u_1 \dotsm X u_j z, v_0 X v_1 \dotsm v_{j - 1} X)
    \end{equation*}
    that has at least all the same nonempty solutions as $E_1$.
    In both cases, $|z| \leq |p|$.
    Like in the case of $E_2$,
    we see that $E_4$ cannot have infinitely many solutions.
    The lengths of all the constant words in $E_4$ are divisible by $|p|$,
    and so are the lengths of at least three nonempty solutions
    (the solutions of~$E_1$).
    We can use Lemma \ref{lem:code} with $Z = \Gamma^{|p|}$ for $E_4$.
    If $h$ is the morphism of Lemma \ref{lem:code},
    then we get the equation
    \begin{equation*}
        E_5: (h(u_0) X h(u_1) \dotsm X h(u_j z),
        h(v_0) X h(v_1) \dotsm h(v_{j - 1}) X).
    \end{equation*}
    It has at least three nonempty solutions, but only finitely many.
    Because $|z| \leq |p|$, $h(u_j z) = h(u) c$,
    where $u$ is a prefix of $u_j$ and $c$ is a letter.
    Because $u_0 \dotsm u_j$ is a prefix of $v_0 \dotsm v_{j - 1}$,
    also $h(u_0 \dotsm u_{j - 1} u) = h(u_0) \dotsm h(u_{j - 1}) h(u)$
    is a prefix of $h(v_0 \dotsm v_{k - 1}) = h(v_0) \dotsm h(v_{k - 1})$.
    We can use Lemma~\ref{lem:empty} with $c$ as $a$ and $m = 1$,
    so $E_5$ has the empty solution.
    This contradicts the minimality of $E_1$.
\end{proof}


\begin{theorem}
    If a system of constant-free three-variable equations
    is independent and has a nonperiodic solution,
    then it has at most 17 equations.
\end{theorem}

\begin{proof}
    Follows from Theorem~\ref{thm:main} and Theorem~\ref{thm:one_three}.
\end{proof}

\section{Conclusion}

We have proved that the maximal size of a finite solution set
of a one-variable word equation is three,
and that the maximal size of an independent system
of constant-free three-variable equations with a nonperiodic solution
is somewhere between two and 17.

Improving the bound 17 is an obvious open problem.
A possible approach would be to improve the results in~\cite{NoSa18ijfcs}.

Another open problem is proving similar bounds
for more than three variables.
The result in~\cite{NoSa18ijfcs} is based on a characterization of
three-generator subsemigroups of a free semigroup
by Budkina and Markov~\cite{BuMa73},
or alternatively a similar result by Spehner~\cite{Sp76,Sp86}.
This means that it is very specific to the three-variable case,
and analyzing the general case would require an entirely different approach.

Finally, characterizing possible solution sets of one-variable equations
would be interesting.
The possible infinite solution sets are given by Theorem~\ref{thm:one_sol_log},
and every singleton set is possible,
but for sets of size two or three the question is open.

\bibliography{../bibtex/ref}

\end{document}